\documentclass[11pt]{article}
\usepackage{amsmath, amssymb, amsthm, hyperref}

\DeclareMathOperator*{\argmin}{arg\,min}

\usepackage{graphicx} 
\usepackage{amsmath, color, comment, mathtools, cleveref, enumitem, authblk}
\usepackage[margin=1.2in]{geometry}

\usepackage[maxbibnames=99]{biblatex}

\bibliography{CPM}

\newtheorem{thm}{Theorem}[section]
\newtheorem{prop}[thm]{Proposition}
\newtheorem{lemma}[thm]{Lemma}

\newtheorem{mydef}[thm]{Definition}

\newtheorem{remark}[thm]{Remark}

\numberwithin{equation}{section} 

\DeclareMathOperator{\sgn}{sgn}

\newcommand{\p}{{\partial}}

\renewcommand{\d}{\delta}

\newcommand{\R}{{\mathbb{R}}}

\newcommand{\Z}{{\mathbb{Z}}}

\newcommand{\T}{{\mathbb{T}}}
\newcommand{\g}{{\mathsf{g}}}

\newcommand{\D}{\Delta}

\newcommand{\nn}{\nonumber}

\newcommand{\ep}{\varepsilon}
\newcommand{\al}{\alpha}
\newcommand{\be}{\beta}

\newcommand{\la}{\lambda}

\newcommand{\Dm}{|\nabla|}

\newcommand{\om}{\omega}

\newcommand{\bmu}{\bar\mu}

\newcommand{\ga}{\gamma}

\let\div\relax
\DeclareMathOperator{\div}{\mathsf{div}}

\def\XXint#1#2#3{{\setbox0=\hbox{$#1{#2#3}{\int}$ }
\vcenter{\hbox{$#2#3$ }}\kern-.6\wd0}}

\def \hal{\frac{1}{2}}
\def\({\left(}
\def\){\right)}
\def \ep{\varepsilon}
\def\nab{\nabla}
\def\indic{\mathbf{1}}
\def\cd{\mathsf{c_{d,s}}}

\usepackage{fancyhdr}

\title{On clogged and fast diffusions in porous media with fractional pressure}
\author{Antonin Chodron de Courcel\thanks{Laboratoire Alexandre Grothendieck, Institut des Hautes Études Scientifiques, Université Paris-Saclay, CNRS, 35 Routes de Chartres, 91440 Bures-sur-Yvette, France; Email: decourcel@ihes.fr}}
\date{January 2025}

\begin{document}
\maketitle

\begin{abstract}
    We study the existence and infinite-speed propagation of solutions to models arising in porous media, when the mobility is highly degenerate (inverse power law). The approach is based on maximum principles for the fractional Laplacian, and allows to derive lower bounds on solutions in a straightforward manner. Finally, in the case of clogged porous media, where the mobility vanishes at points of unbounded density, solutions that become instantaneously bounded are constructed.
\end{abstract}

\section{Introduction.}

This article concerns the existence and some properties of weak solutions to 
\begin{equation}\label{eq:intro}
\begin{cases}
    \p_t\mu -\div(\mu^m\nab\g\ast\mu) =0,\quad (t,x)\in \R_+\times\T^d,\quad m<1, \\
    \mu\vert_{t=0} = \mu_0.
\end{cases}
\end{equation}
Let us describe in detail the main features of the problem. First, the equation is posed on the $d$-dimensional torus ($d\ge 1$), which can be identified with the periodic box $\left[-\hal,\hal\right]^d$. We shall come back to this restriction in the end of this introduction. Second, the interaction potential, $\g$, is a periodized super-Coulombic potential, i.e.
\begin{equation}\label{eq:defg}
     |\nab|^{d-s}\g = \cd (\d_0-1), \quad s\in [(d-2)_+, d), \qquad \cd\coloneqq \begin{cases} \displaystyle\frac{4^{\frac{d-s}{2}}\Gamma((d-s)/2)\pi^{d/2}}{\Gamma(s/2)}, & {-1\leq s<d} \\ \displaystyle\frac{\Gamma(d/2)(4\pi)^{d/2}}{2}, & {s=0}. \end{cases}
\end{equation}
In particular, it can be proved that $\g$ is a zero mean and smooth modification of $x\mapsto|x|^{-s}$ (see \cite{HSSS17}). Although the reader may assume otherwise, the subsequent discussion applies naturally to any smooth modification of this potential. Finally, one special feature of this system is the nonlinearity $\mu^m$. Indeed, \eqref{eq:intro} can be seen as a continuity equation whose flux $\mathsf{j} = \sigma\mu \mathsf{v}$ is proportional to the velocity $\mathsf{v}=-\nab\g\ast\mu$, and where $\sigma$ is the \emph{mobility} of the particles, which serve to model scattering effects of different kinds. When the dominant part of the scattering is due to the particles themselves, the mobility can be taken to be an inverse power law, that is $m<1$. 

\subsection{Related works.}

Several systems related to \eqref{eq:intro} have been studied in past decades. Let us examine some of what is known, open questions, and the aspects on which we focus.

\paragraph{The constant mobility case.}

The case of constant mobility ($m=1$) and super-Coulombic interaction $((d-2)_+<s<d)$ was first studied in the 2010's by Caffarelli and V\'azquez. Starting from a bounded integrable nonnegative datum with exponential decay at infinity, they proved the existence of a solution with finite propagation speed in \cite{CaffarelliVazquez11a}. They also studied the asymptotic behavior of weak solutions, which are attracted by self-similar profiles \cite{CaffarelliVazquez11b}. Together with Soria \cite{CaffarelliSoriaVazquez13, CaffarelliVazquez16}, they proved the H\"older regularity of bounded solutions and an instantaneous $L^\infty$ regularization result of the form
\begin{equation}\label{eq:prototypeLinftybound}
    \forall t>0, \quad \|\mu(t)\|_{L^\infty} \le Ct^{-\d}\|\mu_0\|_{L^1}^\zeta,\quad \d := \frac{d}{s+2},\  \zeta := \frac{s-d+2}{s+2},
\end{equation}
allowing them to build solutions from merely integrable data. This bound was obtained in the spirit of \cite{CaffarelliVasseurAnnMath} and requires an iteration procedure à la Moser. Such a regularization result is now well-known and has been adapted to other situations, in particular to the minimizing movement scheme approach of \cite{AmbrosioSerfaty08} (where $s=0$, $d=2$) in \cite{LisiniMaininiSegatti18}. We finally note that \eqref{eq:prototypeLinftybound} is a consequence of the strong repulsion of particles when $(d-2)_+<s<d$, but is also true in the (less singular) Coulomb case $s=d-2$. This was already known when $d=2$ \cite{LZ00}, with a proof based on the method of characteristics. When $d=1$, Bertucci et al. \cite{BertucciLasryLions2024b} propose a different method to derive this bound, which relies on the Poisson equation $-\D\g = \cd\d_0$ (possibly periodized if we are on the torus). One contribution of this paper is to provide a direct unified method in order to obtain this inequality, as will be clear in the last section. 

Another issue concerns the link with a microscopic model of interacting particles. When the mobility is constant, one can view this equation as a mean-field model for the system
\begin{equation*}
\frac{dX_t^i}{dt} = -\frac{1}{N}\sum_{j\ne i}\nab\g(X_t^i-X_t^j).
\end{equation*}
A rigorous proof of the mean-field convergence for Coulombic interactions is presented in \cite{duerinckx2016mean, serfaty2020mean}, introducing a modulated energy method which requires sufficient regularity. For super-Coulombic interactions, obtaining strong solutions necessitates adding a viscous term and employing a modulated \emph{free} energy method, as detailed in \cite{BJW2019crm, BJW2019edp, BJW2020}. The regularity needed for these approaches is established in \cite{CDCRS2023}. However, when the mobility depends on the density, it is unclear which particle system is appropriate for studying the validity of such equations.

An important difficulty of these systems is the lack of a comparison principle in general. In particular, we cannot use viscosity solutions, and no uniqueness result is known in the super-Coulombic case $(d-2)_+<s<d$. In the Coulomb case $s=d-2$ however, uniqueness of bounded solutions is obtained in \cite{serfatyVazquez2014}. Furthermore, when $d=1$ or if one considers radial solutions, we can integrate in space and map the equation to another one, for which there is a comparison principle.\footnote{ This is a classical link between one-dimensional scalar conservation laws and Hamilton-Jacobi equations.} This strategy is that of \cite{BertucciLasryLions2024b, carrillo2022vortexsupermobility, carrillo2022fastsubmobility}. We finally note that the Riesz potential is displacement convex when $d=1$, which allows \cite{carrillo2015exponential} to exploit the dissipation term by getting a spectral gap lower bound in the spirit of a logarithmic Sobolev inequality, and thus prove exponential convergence towards the equilibrium. 

\paragraph{Density-dependent mobility.}

When the mobility is not constant anymore ($m\neq 1$), the qualitative behavior of solutions heavily depends on the regime studied. 

When $m\in [1,2)$, it is known at least from \cite{STAN2016infinitespeed} that one can build solutions with finite-speed propagation.

The so called ``fast diffusion'' regime $m\in(0,1)$ is very different, as we expect solutions to have infinite speed of propagation. However, this could only be proved in the specific case $d=1$, where one can argue with comparison principles by using an integrated problem \cite{STAN2014noteInfinitepropagation, STAN2016infinitespeed}. One contribution of this article is to prove an infinite speed of propagation result in any dimension, for non necessarily radial solutions on the torus. 

In \cite{stanTesoVazquez2019existence}, the authors provide an existence result for any $m>0$, and super-Coulombic interactions $(d-2)_+<s<d$, starting from nonnegative measure initial data of finite mass. Their proof relies on successive compactness arguments and an $L^1-L^\infty$ smoothing effect obtained from Moser iterations, which we have already discussed. 

Concerning the Coulomb case $s=d-2$, we refer to \cite{carrillo2022fastsubmobility} for the fast diffusion case $m<1$ and to \cite{carrillo2022vortexsupermobility} for the case $m>1$ (where compactly supported solutions are obtained). In both cases, the authors rely on an integrated problem by considering radial solutions, thus making use of the link between the scalar conservation law and its Hamilton-Jacobi equation. Again, a comparison principle holds  at the level of the Hamilton-Jacobi equation, which allows to use viscosity solutions. In contrast, the qualitative properties of non radial solutions is quite an open problem. Most of what shall be presented in this article extends to the Coulomb case $s=d-2$. However, strong compactness estimates is a delicate question here, so that the main results only concern the super-Coulombic case.

Finally, there is no existence theory for the clogged porous medium model $m<0$. Nevertheless, \cite{stan2015transformations} constructs self-similar solutions in this case. These solutions may vanish in finite time if $\frac{-md}{s-d+2}>1$. One contribution of this article is to develop an existence theory for any $m<0$. 

We finally note that other kinds of inhomogeneities can be considered. An example includes a non-linearity in Darcy's law in the form
\begin{equation}\label{eq:nonlinearpressure}
    \p_t \mu -\div (|\mu| \nab\g\ast(|\mu|^{m-2}\mu) ) =0.
\end{equation}
This model has been studied in \cite{biler2015nonlocal} for $m>1$, where explicit self-similar profiles and weak solutions are obtained. In the case of positive solutions, finite speed of propagation is obtained in \cite{Imbert14} and, for $m>2$, uniqueness and H\"older regularity of solutions is proved in \cite{DaoDiaz2025}. It is also possible to mix such kinds of inhomogeneities with a nonlinear mobility, which leads to the following model:
\begin{equation*}
    \p_t \mu -\div (|\mu|^{m_1} \nab\g\ast(|\mu|^{m_2-1}\mu) ) =0.
\end{equation*}
In \cite{nguyen2018porous}, the authors study this quite general equation in the regime $m_1, m_2>0$, and provide a universal bound in the form of\footnote{Obviously, the constants $\d,\zeta$ will not be the same.} \eqref{eq:prototypeLinftybound} when $m_1+m_2>1$. 

\bigbreak

We end this introduction by noting an important difference between the nonlinear mobility regime and the nonlinear pressure law regime described by \eqref{eq:nonlinearpressure}. Indeed, equation \eqref{eq:intro} can be viewed as a transport–diffusion equation with effective velocity $-\mu^{m-1}\nabla \g\ast\mu$. Consequently, if the solution is steep near some point, the velocity may differ between the top and bottom of this steepness by an amount of order $1$. This can lead to the formation of shocks, as expected in the Coulomb case $s=d-2$. When $s>d-2$, it is unclear whether diffusion dominates this mechanism; the question of higher regularity remains open, although an argument à la De Giorgi might apply.

\subsection{Main results.}

Let us now state the main results proven in this article. The first result concerns the clogged diffusion model:
\begin{equation}\label{eq:clogged}
\begin{cases}
    \p_t\mu-\div(\mu^{-m}\nab\g\ast\mu) = 0,\quad m>0,\\
    \mu\vert_{t=0} = \mu_0.
\end{cases}
\end{equation} 
The second result concerns the fast diffusion case:
\begin{equation}\label{eq:FD}
\begin{cases}
    \p_t\mu-\div(\mu^m\nab\g\ast\mu) = 0,\quad m\in (0,1),\\
    \mu\vert_{t=0} = \mu_0.
\end{cases}
\end{equation}
Let us define the usual notion of weak solution for these problems. 

\begin{mydef}
    Let $d\ge 1$, $(d-2)_+<s<d$, $m\in \R$, and $\mu_0\in \mathcal{M}_+(\T^d)$. A \emph{weak solution} to \eqref{eq:intro} is a distribution $\mu$ on $[0,\infty)\times\T^d$ such that $\nab\g\ast\mu\mu^{m}\in L^1_{loc}([0,\infty)\times\T^d)$, $\mu\in L^\infty_{loc}((0,\infty), L^1(\T^d))$, and, for all $\phi\in C^{0,1}([0,\infty)\times\T^d)$ and $T>0$, 
    \begin{multline*}
        \int_0^T\int_{\T^d}\p_t\phi(t,x) \mu(t,x)  dxdt + \int_{\T^d}\phi(T,x)\mu(T,x)dx - \int_{\T^d}\phi(0,x)d\mu_0(x) \\
        + \int_0^T\int_{\T^d}\nab\phi(t,x)\cdot \nab\g\ast\mu(t,x) \mu^{m}(t,x) dxdt = 0.
    \end{multline*}
\end{mydef}
The first main result of this paper is:
\begin{thm}[Clogged diffusion]\label{thm:clogged}
    Let $d\ge 1$, $(d-2)_+ < s<d$, and $m>0$. Define
    \begin{equation*}
        \ga := \frac{md}{s-d+2}.
    \end{equation*}
    If $\ga \ge 1$, consider $\mu_0 \in L^{p}(\T^d)\cap \dot H^\frac{s-d}{2}(\T^d)$, for some $ p>\ga$. Otherwise, we can take $\mu_0\in \mathcal{M}_+(\T^d)\cap \dot H^\frac{s-d}{2}(\T^d)$. Assume finally $\mu_0\ge 0$. Then, there exists a weak solution $\mu$ to \eqref{eq:clogged} satisfying for all $0<\tau<t$,
    \begin{enumerate}[label=$\left(\roman*\right)$]
        \item  (decrease of the $L^1$ norm) $$\displaystyle\|\mu(t)\|_{L^1} \le \|\mu(\tau)\|_{L^1},$$
        \item (conservation of mass) $$\displaystyle\int_{\T^d}\mu(t,x)dx = \int_{\T^d} d\mu_0 =: \bmu,$$ 
        \item (weak maximum principle) $t\mapsto \mathrm{ess\ inf\ } \mu(t)$ is nondecreasing and $t\mapsto \mathrm{ess\ sup\ }\mu(t)$ is nonincreasing,
        \item (dissipation of the $L^p$ norms)\footnote{We use the convention that $\displaystyle\frac{x^0}{0} = \log x$ in the case $p=m+1$.} for all $p\in(1, \infty)$, \begin{equation*}
        \displaystyle\|\mu(t)\|_{L^p}^p + \frac{p(p-1)}{p-m-1}\cd \int_\tau^t \int_{\T^d} \mu^{p-1-m}|\nab|^{s-d+2}\mu  dxd\tau' \le \|\mu(\tau)\|_{L^p}^p,
        \end{equation*}
        
        \item  (energy dissipation) \begin{equation*}
            \displaystyle\int_{\T^d} \g\ast\mu(t)\mu(t)dx + \int_\tau^t\int_{\T^d}|\nab\g\ast\mu|^2 \mu^{-m} dxd\tau' \le \int_{\T^d}\g\ast\mu(\tau)\mu(\tau)dx,
        \end{equation*}
        
        \item (infinite-speed propagation) $$\mu(t,x)\ge \Phi(t),\quad \text{for a.e. } t>0 \text{ and } x\in \T^d, $$ where $\Phi$ is the unique solution to
    \begin{equation}\label{thm:ODE}
    \begin{cases}
        \displaystyle\dot \Phi = C_{d,s} \frac{\bmu-\Phi}{\Phi^m},\\
        \Phi(0)=0.
    \end{cases}
    \end{equation}
    \item (instantaneous regularization) $$\displaystyle\|\mu(t)\|_{L^\infty} \le C_{s,p,d} t^{-\d_p}\|\mu_0\|_{L^p}^{\zeta_p} + \omega( \bmu ),$$ where
    \begin{equation*}
        \displaystyle \d_p := \frac{d}{p(s-d+2)-md},\quad \zeta_p := \frac{(|p-m|+m)(s-d+2)}{(|p-m|+m)(s-d+2)-md},
    \end{equation*}
    and $\omega:\R_+\to\R_+$ is some continuous increasing map depending on $s,d,p$, such that $\omega(0)=0$.
    \end{enumerate}
    
\end{thm}
\begin{remark}\label{rem:introLowerbound}
    The lower barrier given by \eqref{thm:ODE} implies that for almost every $t>0$ and $x\in \T^d$,
    \begin{equation*}
        \mu(t,x)\ge C\min\bigg((\bmu t)^\frac{1}{1+m}, \quad \bmu \left(1-e^{-C_{d,s}t}\right) \bigg),
    \end{equation*}
    for some $C>0$ depending on $s,d,m$. The scaling law $t^{\frac{1}{1+m}}$ is the same as in \cite{stan2015transformations}, for type I self-similar solutions (obtained in the range $\ga<1$). In contrast, when $\ga>1$, the authors present (type II) self-similar solutions that vanish in finite time due to mass escaping to infinity. However, since we are working within a bounded domain, this mass leakage is avoided.
\end{remark}
\begin{remark}
    Point $(vii)$ may be surprising. Indeed, the mobility is degenerate such that the velocity vanishes at points where $\mu = +\infty$. In particular, we could have imagined that solutions would remain unbounded. It is nevertheless enough to assume some integrability on the initial data.
\end{remark}
\begin{remark}
    On the Euclidean space $\R^d$, we would expect solutions to generically vanish in finite time for some large enough $m>0$. This has been proved for some self-similar solutions in \cite{stan2015transformations}. This is why we restrict ourselves to a confined domain. 
\end{remark}
\begin{remark}
    The restriction to finite energy initial datum $\dot H^\frac{s-d}{2}$ ensures weak continuity, as will be clear at the end of Section 3. This assumption can be refined. 
\end{remark}

\begin{remark}
    We could have considered the model
    \begin{equation*}
        \p_t \mu - \div(|\mu|^{-m}\nab\g\ast u) = 0,
    \end{equation*}
    for a signed initial data $\mu_0$, which we leave as a possible extension of our work.
\end{remark}

The second main result of this paper is:
\begin{thm}[Fast diffusion]\label{thm:FD}
    Let $d\ge 1$, $(d-2)_+ < s<d$, and $m\in(0,1)$. Consider $\mu_0\in \mathcal{M}_+(\T^d)$. There exists a solution $\mu$ to \eqref{eq:FD} satisfying for all $0<\tau<t$,
    \begin{enumerate}[label=$\left(\roman*\right)$]
    \item (decrease of the $L^1$ norm) $$\displaystyle\|\mu(t)\|_{L^1} \le \|\mu(\tau)\|_{L^1},$$ 
        \item (conservation of mass) $$\displaystyle\int_{\T^d}\mu(t)dx = \int_{\T^d}d\mu_0 =: \bmu,$$
        \item  (weak maximum principle) $t\mapsto \mathrm{ess\ inf\ } \mu(t)$ is nondecreasing and $t\mapsto \mathrm{ess\ sup\ } \mu(t)$ is nonincreasing,
        \item (dissipation of the $L^p$ norms) for all $p\in(1, \infty)$,    
        \begin{equation*}
        \displaystyle\|\mu(t)\|_{L^p}^p + \frac{p(p-1)}{p+m-1}\cd \int_\tau^t \int_{\T^d} \mu^{p-1+m}|\nab|^{s-d+2} \mu\ dxd\tau' \le \|\mu(\tau)\|_{L^p}^p,
        \end{equation*}
        
        \item (energy dissipation) 
        \begin{equation*}
            \displaystyle\int_{\T^d} \g\ast\mu(t)\mu(t)dx + \int_\tau^t\int_{\T^d}|\nab\g\ast\mu|^2 \mu^{m}\ dxd\tau' \le \int_{\T^d}\g\ast\mu(\tau)\mu(\tau)dx,
        \end{equation*}
         
        \item (infinite-speed propagation) 
        $$\mu(t,x)\ge \Phi(t),\quad \text{for a.e. } t>0 \text{ and } x\in \T^d,$$ where $\Phi$ is the unique \emph{positive} solution to
    \begin{equation}\label{thmFD:ODE}
    \begin{cases}
        \displaystyle\dot \Phi = C_{d,s} \Phi^m \left(\bmu-\Phi\right),\\
        \Phi(0)=0.
    \end{cases}
    \end{equation}
    \item (instantaneous regularization) $$\|\mu(t)\|_{L^\infty} \le C_{s,d}\min(1,t)^{-\d}\bmu,$$
    where $\displaystyle \d := \frac{d}{s-d+2 - md}.$
    \end{enumerate}
\end{thm}
\begin{remark}
    The lower barrier \eqref{thmFD:ODE} implies that for almost every $t>0$ and $x\in \T^d$,
    \begin{equation*}
        \mu(t,x)\ge C\min\bigg( (\bmu t)^\frac{1}{1-m}, \quad \bmu \left(1-e^{-C_{d,s}t}\right) \bigg),
    \end{equation*}
    for some $C>0$ depending on $s,d,m$. The scaling law $t^{\frac{1}{1-m}}$ is again the same as in \cite{stan2015transformations}.
\end{remark}
\begin{remark}
    As already mentioned, \cite{stanTesoVazquez2019existence} constructed weak solutions to this equation. The main novelty here is the lower barrier and the straightforward proof of the $L^\infty$ regularization result.
\end{remark}

\subsection{Acknowledgments.}

The author thanks the anonymous reviewers for their careful reading of the first draft of this manuscript. He also thanks Laure Saint-Raymond and Sylvia Serfaty for inspiring and insightful discussions, and acknowledges support from the Fondation CFM, through the Jean-Pierre Aguilar fellowship.

\section{Key estimates.}

In this section, we provide some key estimates for the fractional Laplacian which hold without assuming regularity of the function. These are maximum/minimum nonlinear principles for fractional diffusions.

\begin{lemma}\label{lem:minprinciple}
    Let $\al\in (0,2)$ and $d\geq 1$. Consider a smooth bounded map $\mu:\T^d\to \R$ and a point $\bar x$ where $\mu$ reaches its minimum. Then,
    \begin{equation*}
        (-\D)^{\frac{\al}{2}} \mu(\bar x) \le C\left( \mu(\bar x)-\bmu\right),
    \end{equation*}
    where $\displaystyle \bmu := \int_{\T^d} \mu dx$ and $C>0$ depends only on $d$ and $\al$.
\end{lemma}
\begin{remark}
    Note that the constant $C$ above a priori depends on the size of the domain, as the following proof shows.
\end{remark}
\begin{proof}
    Consider the periodic rewriting of the fractional Laplacian on the torus (see e.g. \cite[Proposition 2.2]{cordoba2004maximum})\footnote{This reference considers the case $d=2$, but this formula clearly extends to general dimensions.}. Since $\bar x$ is a point of minimum, we have $\mu(\bar x) \le \mu$ on $\T^d$. Therefore, by Tonelli's theorem we can swap the sum and integral below.
    \begin{align*}
        &(-\D)^\frac{\al}{2} \mu(\bar x)  \\
        &= c_{d,\al} \sum_{k\in \Z^d} \int_{\T^d} \frac{\mu(\bar x)-\mu(y)}{|\bar x-y-k|^{d+\al}} dy  \\
        &= c_{d,\al} \int_{\T^d}  \frac{\mu(\bar x)-\mu(y)}{|\bar x-y|^{d+\al}} dy + c_{d,\al} \sum_{k\neq 0}  \int_{\T^d} \frac{\mu(\bar x)-\mu(y)}{|\bar x-y-k|^{d+\al}} dy   \\
        &\le c_{d,\al}   \int_{\T^d} \frac{\mu(\bar x)-\mu(y)}{|\bar x-y|^{d+\al}} dy,  
    \end{align*}     
    where we have selected the most singular contribution by dropping the nonpositive terms involving $|k|\neq 0$. Observing that\footnote{The distance considered here is that of the torus, defined by $|x| := \inf_{k\in \Z^d} |x-k|$.} $|\bar x-y| \le  2^{-1}\sqrt{d}$, we obtain since $|\T^d|= 1$
    \begin{align*}
        (-\D)^\frac{\al}{2} \mu(\bar x) &\le c_{d,\al} \int_{\T^d} (\mu(\bar x)-\mu(y))dy \\
        &= c_{d,\al}(\mu(\bar x) -\bmu ).
    \end{align*}
\end{proof}

We now provide the maximum principle counterpart, which is a generalization of the Constantin-Vicol maximum principle \cite{ConstantinVicol12}.

\begin{lemma}\label{prop:pwD}
    Let $d\geq 1$ and $\al \in (0,2)$. Let $\mu : \T^d\to \R$ be a smooth nonnegative function and $h:\R\to \R$ nondecreasing and convex such that $h(0)= 0$. Then, there exists constants $C,C'>0$ depending only on $\al$ and $d$ such that, given any point $\Bar{x}$ where $\mu$ reaches its maximum, either
        \begin{equation*}
        (-\D)^\frac{\al}{2} h(\mu)(\Bar{x}) \geq C h(\mu(\Bar{x})) \|\mu\|_{L^\infty}^{\frac{ \al}{d}} \|\mu \|_{L^1}^{-\frac{\al}{d}},
        \end{equation*}
    or
    \begin{equation*}
        \|\mu\|_{L^\infty} \le C'\|\mu\|_{L^1}.
    \end{equation*}
\end{lemma}
\begin{remark}
    Such a pointwise inequality for the fractional Laplacian has been derived in the context of Burgers and the surface quasi-geostrophic equation \cite{ConstantinVicol12}. Here, we provide an inequality that does not depend on antiderivatives of $\mu$. The extension to nonlinear (convex) $h:\R\to\R$ is given for free, and may be useful in some other equations, e.g. of the form
    \begin{equation*}
        \p_t u -\div (u^{m_1}\nab\g\ast u^{m_2})= 
        0,
    \end{equation*}
    $m_2>0$ taking into account some sort of inhomogeneities, as in \cite{biler2015nonlocal, nguyen2018porous}.
\end{remark}
\begin{proof}
    We have 
    \begin{equation*}
        (-\D)^\frac{\al}{2} h(\mu)(\Bar{x}) = c_{d,\al}\int_{\T^d} \sum_{k\in\Z^d}\frac{h(\mu(\Bar{x}))-h(\mu(y))}{|\Bar{x}-y-k|^{d+\al}} dy.
    \end{equation*}
    Since $h$ is nondecreasing and $\bar x$ is a point of maximum for $\mu$, the integrand above is always nonnegative. Moreover, the most singular contribution is given by $k=0$. Fixing $0<R<2^{-\hal}\sqrt{d}$ momentarily, one obtains
    \begin{align*}
        (-\D)^\frac{\al}{2} h(\mu)(\bar x) &\geq c_{d,\al} \int_{B(\bar x, R)^c} \frac{h(\mu(\Bar{x}))-h(\mu(y))}{|\Bar{x}-y|^{d+\al}} dy \\
        &= c_{d,\al} h(\mu(\bar x)) \int_{B(\bar x, R)^c} \frac{dy}{|\bar x-y|^{d+\al}} - c_{d,\al}  \int_{B(\bar x, R)^c} \frac{h(\mu(y))}{|\bar x-y|^{d+\al}} dy \\
        &= c_{d,\al}' h(\mu(\bar x))  \frac{1}{ R^\al} - c_{d,\al} \int_{B(\bar x, R)^c} \frac{h(\mu(y))}{|\bar x-y|^{d+\al}} dy.
    \end{align*}
    Now, since $h(0)=0$ and $h$ is convex, we have
    \begin{equation*}
        h(\mu(y)) \le \frac{\mu(y)}{\|\mu\|_{L^\infty}} h(\|\mu\|_{L^\infty}).
    \end{equation*}
    Thus,
    \begin{align*}
        (-\D)^\frac{\al}{2} h(\mu)(\bar x) &\ge  c_{d,\al}' h(\mu(\bar x))  \frac{ 1}{ R^{\al}} - c_{d,\al} \frac{h(\mu(\bar x))}{\|\mu\|_{L^\infty}} \int_{B(\bar x, R)^c} \frac{\mu(y)}{|\bar x-y|^{d+\al} } dy \\
        &\ge c_{d,\al}' h(\mu(\bar x))  \frac{ 1}{ R^{\al}} - c_{d,\al} \frac{h(\mu(\bar x))}{\|\mu\|_{L^\infty}} \|\mu\|_{L^1} \sup_{y\in B(\bar x, R)^c} \frac{1}{|\bar x-y|^{d+\al} } \\
        &=  h(\mu(\bar x)) \left( c_{d,\al}'\frac{ 1}{R^{\al}} -c_{d,\al} \frac{1}{\|\mu\|_{L^\infty}} \|\mu\|_{L^1} \frac{1}{R^{d+\al}} \right).
    \end{align*}
    Taking 
    \begin{equation*}
        R = \left(\frac{(d+\al) c_{d,\al} \|\mu\|_{L^1}}{c_{d,\al}' \|\mu \|_{L^\infty}}\right)^\frac{1}{d}
    \end{equation*}
    gives the result. However, this is possible only if 
    \begin{equation*}
        \left(\frac{(d+\al)c_{d,\al} \|\mu\|_{L^1}}{c_{d,\al}' \|\mu \|_{L^\infty}}\right)^\frac{1}{d} < \hal\sqrt{d}
    \end{equation*}
    in general. If this is not possible, then the other part of the proposition holds. 
\end{proof}

Finally, we will need the following lemma in order to obtain bounds in a straightforward manner.

\begin{lemma}\label{lem:mincurve}
    Let $T>0$ and $\mu\in C^1([0,T]\times \T^d)$ and consider a curve $\bar x: [0,T]\to \T^d$ such that, for all $t\in [0, T]$, $\bar x(t)$ is a minimum (resp. maximum) point of $\mu(t,\cdot)$. Assume moreover that $t\mapsto \min \mu(t,\cdot)$ is nondecreasing (resp. $t\mapsto \max \mu(t,\cdot)$ is nonincreasing). Then, $t\mapsto \mu(t,\bar x(t))$ is differentiable almost everywhere and, for almost every $t>0$,
    \begin{equation*}
        \frac{d}{dt}\mu(t,\bar x(t)) = \p_t \mu(t, \bar x(t)).
    \end{equation*}
\end{lemma}
\begin{remark}
    We assume neither regularity nor uniqueness on the curve $\bar x$.
\end{remark}
\begin{proof}
    Let us assume that $t\mapsto \min \mu(t,\cdot)$ is nondecreasing. In particular, it is therefore differentiable almost everywhere. Moreover,
    \begin{align*}
        \mu(t+dt,x) = \mu(t,x) + dt \p_t\mu(t,x) + o(dt),
    \end{align*}
    uniformly in $x$. Applying this to $\bar x(t)$ gives for all $t>0$ 
    \begin{equation*}
        \mu(t+dt, \bar x(t+dt)) \le \mu(t+dt,\bar x(t)) = \mu(t,\bar x(t)) + dt\p_t\mu(t,\bar x(t)) + o(dt).
    \end{equation*}    
    Dividing by $dt$, we get
    \begin{align*}
    \begin{cases}
       \displaystyle   \frac{\mu(t + dt, \bar x(t+dt))- \mu(t,\bar x(t))}{dt}\le \p_t \mu(t,\bar x(t)) + o(1),  & dt>0, \\
      \displaystyle  \frac{\mu(t + dt, \bar x(t+dt))- \mu(t,\bar x(t))}{dt}\ge \p_t \mu(t,\bar x(t)) + o(1), & dt<0.
    \end{cases}
    \end{align*}
    Sending $dt\to 0$ at points $t>0$ where $t\mapsto \min \mu(t,\cdot)$ is differentiable, we obtain 
    \begin{equation*}
        \frac{d}{dt} \mu(t,\bar x(t)) \le \p_t\mu(t,\bar x(t)) \le \frac{d}{dt}\mu(t,\bar x(t)).
    \end{equation*}
    The case where $t\mapsto \max \mu(t,\cdot)$ is nonincreasing follows immediately from the same argument.
\end{proof}

\section{Clogged diffusions.}

In this section, we study clogged diffusions, i.e. models of the form
\begin{equation*}
\begin{cases}
    \p_t\mu - \div (\mu^{-m}\nab\g\ast\mu) =0, \\
    \mu\vert_{t=0} = \mu_0 \ge 0,
\end{cases}
\end{equation*}
where $m>0$. Of course, a major problem is the singular behavior of the mobility $\mu^{-m}$ where $\mu$ vanishes. The key ingredient is therefore to obtain an a priori estimate of the form of an integrable lower barrier on the solution. Moreover, such a mobility is degenerate at points where $\mu=+\infty$, so that unbounded initial data may not regularise into bounded solutions at positive times. 

\subsection{Existence theory.}

In order to construct solutions to \eqref{eq:clogged}, we will consider the following approximate problem
\begin{equation}\label{eq:approx}\begin{cases}
    \p_t\mu_\ep -\div (|\mu_\ep|^{-m}\nab\g\ast\mu_\ep ) = \ep\D\mu_\ep, \\
    \mu_\ep\vert_{t= 0} = \mu_{0,\ep},
\end{cases}
\end{equation}
for $\ep>0$, where $\mu_{0,\ep}$ is smooth, bounded and satisfies $\mu_{0,\ep} \ge \ep$ and $$\int_{\T^d} \g\ast \mu_{0,\ep}\mu_{0,\ep}dx\xrightarrow[\ep\to 0]{} \int_{\T^d} \g\ast\mu_0  d\mu_0.$$ Therefore, there exists a solution $\mu_{\ep}\in C^{1,1}([0,\infty)\times \T^d)$ to \eqref{eq:approx}. In what follows, all the a priori estimates are independent of $\ep\ge 0$. 

We will need the following inequality, which is very classical in the Euclidean setting $\R^d$. Nevertheless, we could not find a reference on the torus, so we include it below.

\begin{lemma}(Stroock-Varopoulos inequality)\label{lem:SVI}
    Let $d\ge 1$, $p\in \R$, $\gamma\in (0,2)$, and $\mu\in C^\infty(\T^d)$ positive. The following inequality holds.
    \begin{equation}
        \frac{1}{p-1}\int_{\T^d} \mu^{p-1}(-\D)^\frac{\gamma}{2}\mu \ dx \ge \frac{2}{p^2} \int_{\T^d} |(-\D)^\frac{\gamma}{4} \mu^\frac{p}{2}|^2 dx\label{ineq:SVI},
    \end{equation}
    with the convention $\displaystyle \frac{1}{p-1}\mu^{p-1} = \log\mu$ if $p=1$. 
\end{lemma}
\begin{proof}
    We first record the following inequality: for all $t, s\ge 0$ and $p\ge 1$, 
    \begin{equation*}
        \frac{4(p-1)}{p^2}(t^\frac{p}{2} - s^\frac{p}{2})^2 \le (t-s)(t^{p-1}-s^{p-1}).
    \end{equation*}
    This can be found e.g. in \cite{liskevich1993some}. In the special case $p=1$, we have
    \begin{align*}
        4(t^\frac{1}{2}-s^\frac{1}{2})^2 &=\left( \int_s^t u^{\frac{1}{2}-1}\right)^2 \\
        &\le \left\vert \int_s^t du\right\vert \left\vert \int_s^t u^{-1}du\right\vert \\
        &= |t-s| |\log(t)-\log(s)| \\
        &= (t-s)(\log(t)-\log(s)),
    \end{align*}
    the last line being obtained by increasingness of the logarithm. Now, we use the following rewriting of the fractional Laplacian on the torus.
    \begin{align*}
        \int_{\T^d} \mu^{p-1} |\nab|^{\gamma}\mu\ dx &= c_{d,\ga}\int_{\T^d}\mu(x)^{p-1} \int_{\T^d}\sum_{k\in \Z^d} \frac{\mu(x)-\mu(y)}{|x-y-k|^{d+\gamma}} dydx \\
        &= \hal c_{d,\ga} \int_{\T^d}\int_{\T^d}  (\mu(x)^{p-1} -\mu(y)^{p-1})\sum_{k\in \Z^d}\frac{\mu(x)-\mu(y)}{|x-y-k|^{d+\gamma}}dydx,
    \end{align*}
    where we have swapped the variables in order to obtain the last line. Use the previous inequality to obtain\footnote{The special case $p=1$ follows accordingly, with again the convention that $\frac{x^0}{0} = \log x$.}
    \begin{equation*}
        \frac{1}{p-1}\int_{\T^d} \mu^{p-1} |\nab|^{\gamma}(\mu-\bmu) dx \ge \frac{2}{p^2}c_{d,\ga}\int_{\T^d}\int_{\T^d}\sum_{k\in\Z^d}\frac{|\mu(x)^{\frac{p}{2}} - \mu(y)^\frac{p}{2}|^2}{|x-y-k|^{\gamma +d}} dydx.
    \end{equation*}
    The last term is precisely the $\dot H^\frac{\gamma}{2}$ Gagliardo seminorm of $\mu^\frac{p}{2}$.

    We now prove the case $p<1$. As before, we record the following inequality: for all $t,s>0$ and $p< 1$,
    \begin{equation*}
        (t-s)(t^{p-1}-s^{p-1}) \le \frac{4(p-1)}{p^2}(t^\frac{p}{2}-s^\frac{p}{2})^2.
    \end{equation*}
    We proceed as in \cite{liskevich1993some}:
    \begin{align*}
        \frac{4}{p^2}(t^\frac{p}{2}-s^\frac{p}{2})^2 &=\left( \int_s^t u^{\frac{p}{2}-1}\right)^2 \\
        &\le \left\vert \int_s^t du\right\vert \left\vert \int_s^t u^{p-2}du\right\vert \\
        &= |t-s| \frac{|t^{p-1}-s^{p-1}|}{|p-1|} \\
        &= (t-s)\frac{(t^{p-1}-s^{p-1})}{p-1},
    \end{align*}
    where we have used that $t\mapsto t^{p-1}$ is decreasing. We then proceed as in the case $p\ge 1$.
\end{proof}
We now prove the following estimates. 
\begin{prop}\label{prop:classicalaprioriClogged}
    Let $T>0$ and $\mu_\ep\in C^{1,1}([0,T]\times \T^d)$ be a strong solution to \eqref{eq:approx} such that $\mu_{0,\ep} \ge \ep>0$. Then, the following estimates hold for all $t\ge 0$.
    \begin{enumerate}[label=$\left(\roman*\right)$]
    \item (decrease of the $L^1$ norm) $$\displaystyle\|\mu_\ep(t)\|_{L^1} \le \|\mu_{0,\ep}\|_{L^1},$$ 
        \item (conservation of mass) $$\displaystyle\int_{\T^d}\mu_\ep(t)dx = \int_{\T^d}d\mu_{0,\ep} =: \bmu_\ep,$$
        \item  (weak maximum principle) $t\mapsto \min \mu_\ep(t)$ is nondecreasing and $t\mapsto \max \mu_\ep(t)$ is nonincreasing,
        \item (dissipation of the $L^p$ norms) for all $p\in(1, \infty)$,    
        \begin{equation}\label{ineq:Lpdissipclogged}
        \displaystyle\|\mu_\ep(t)\|_{L^p}^p + \frac{p(p-1)}{p-m-1}\cd \int_0^t \int_{\T^d} \mu_\ep^{p-1-m}|\nab|^{s-d+2} \mu_\ep \ dxd\tau \le \|\mu_{0,\ep}\|_{L^p}^p,
        \end{equation}
        
        \item (energy dissipation) 
        \begin{equation}\label{ineq:energydissipationClogged}
            \displaystyle\int_{\T^d} \g\ast\mu_\ep(t)\mu_\ep(t)dx + \int_0^t\int_{\T^d}|\nab\g\ast\mu_\ep|^2 \mu_\ep^{-m}\ dxd\tau \le \int_{\T^d}\g\ast\mu_\ep(\tau)\mu_\ep(\tau)dx.
        \end{equation}
        
    \end{enumerate}
\end{prop}
\begin{remark}
    Combining the $L^p$ dissipation estimate \eqref{ineq:Lpdissipclogged} with the Stroock-Varopoulos inequality \eqref{ineq:SVI} yields, for all $p\in (1, \infty)$ and all $t\ge 0$,
    \begin{equation}\label{ineq:Lp+SVI}
        \|\mu_\ep(t)\|_{L^p}^p + \frac{p(p-1)}{(p-m)^2}\cd \int_0^t\int_{\T^d} ||\nab|^\frac{s-d+2}{2}\mu_\ep^\frac{p-m}{2}|^2dxd\tau \le \|\mu_{0,\ep}\|_{L^p}^p. 
    \end{equation}
\end{remark}
\begin{proof} We omit the $\ep$ subscript in the proof, so that it is more readable.

\bigbreak
\noindent \textbf{Proof of $(i)$ and $(ii)$.}
    We have
    \begin{align*}
        \frac{d}{dt}\int_{\T^d}\mu(t)dx = \int_{\T^d} \div(|\mu|^{-m}\nab\g\ast \mu) dx+ \ep \int_{\T^d}\D\mu dx 
        = 0,
    \end{align*}
    so that the mass is conserved. The decreasing of the $L^1$ norm is obtained by writing $\mu := \lim_{\d\to 0} \sqrt{\mu^2 + \d^2}$, integrating by parts, and letting $\d\to 0$. 
\bigbreak
\noindent \textbf{Proof of $(iii)$.}    
    What preceedes implies that, if $\mu_0 \geq 0$, then $\mu(t) \geq 0$ on its lifespan: indeed, let $\mu_+(t), \mu_-(t)$ denote the positive and negative parts of $\mu(t)$ respectively. Then
\begin{equation*}
    \frac{d}{dt}\int_{\T^d} \mu_\pm(t) dx = \hal\frac{d}{dt}\int_{\T^d} \left(|\mu(t)| \pm \mu(t) \right)dx \leq 0,
\end{equation*}
where the final inequality follows from the first two assertions. In particular, if $\int_{\T^d} (\mu_0)_- dx =0$, then $0\le \int_{\T^d} \mu_-(t) dx \le \int_{\T^d} (\mu_0)_- dx =0$, which implies that $\mu(t)_- = 0$. The same reverse reasoning works as well: if $\mu_0 \le 0$, then $\mu(t) \le 0$ on its lifespan. From now on, we suppose that $\mu_0\ge 0$, so that the equation reads
\begin{equation*}
     \p_t\mu - \div(\mu^{-m}\nab\g\ast \mu) = \ep\D\mu.
\end{equation*}
Let us now prove that $\inf\mu(t)\ge \inf\mu_0$ and $\sup\mu(t) \le \sup\mu_0$ for all $t>0$. Consider $0\leq c\leq \int_{\T^d}\mu_0$ and use integration by parts and \eqref{eq:defg} to compute,\footnote{Recall that we have defined $\bar\mu := \int_{\T^d}\mu_0dx$ to be the mass, which is conserved.}
\begin{align}
\frac{d}{dt}\int_{\T^d}(\mu-c)_{-}dx &= -\int_{\mu(t)\leq c}(\div(\mu(|\mu|+\ep)^{-m-1}\nabla\g\ast\mu) +\ep\D\mu)dx\nn\\
&=- \int_{\mu(t)\leq c}\nab(\mu(|\mu|+\ep)^{-m-1})\cdot\nabla\g\ast\mu dx + \cd\int_{\mu(t)\leq c}\mu(|\mu|+\ep)^{-m-1}\Dm^{2+s-d}(\mu-\bmu) dx \nn\\
&\quad -\ep\int_{\mu(t)=c}\nabla\mu\cdot \frac{\nabla\mu}{|\nabla\mu|}dx \label{eq:dtmucrhs}.
\end{align}
The last term is negative and can be discarded. For the first term, observe that $\nab (\mu(|\mu|+\ep)^{-m-1} )\indic_{\mu \le c}= -m (|\mu|+\ep)^{-m-1} \sgn(\mu) \nab \mu = \nab (\mu(|\mu | + \ep)^{-m-1}-c(c+\ep)^{-m-1})_+$ $\nabla\mu^{-m}\indic_{\mu(t)\leq c} = \nabla(\mu^{-m}-c^{-m})_{+}$ a.e. Hence, integrating by parts,
\begin{align*}
-\int_{\mu(t)\leq c}\nabla\mu^{-m} \cdot\nabla\g\ast\mu dx= -\cd\int_{\T^d}(\mu^{-m}-c^{-m})_+\Dm^{2+s-d}(\mu-\bar \mu)dx.
\end{align*}
Similarly, $\mu^{-m}\indic_{\mu\leq c} = (\mu^{-m}-c^{-m})_{+} +c^{-m}\indic_{\mu\leq c}$, which implies that the right-hand side of \eqref{eq:dtmucrhs} is less or equal to
\begin{align}\label{eq:rhsL1muc}
  \cd c^{-m}\int_{\mu(t)\leq c}\Dm^{2+s-d}(\mu-\bar\mu)dx.
\end{align}
We use the definition of the fractional Laplacian on $\T^d$ again to write
\begin{align*}
\int_{\mu(t)\leq c}\Dm^{2+s-d}(\mu-\bar\mu)dx = C_{s,d}\sum_{k\in\Z^d} \int_{\mu(t)\leq c}\int_{\T^d} \frac{\mu(t,x)-\mu(t,y)}{|x-y-k|^{d+(2+s-d)}}dydx.
\end{align*}
Evidently,
\begin{equation*}
\int_{\mu(t)\leq c}\int_{\mu(t)>c}\frac{\mu(t,x)-\mu(t,y)}{|x-y-k|^{d+(2+s-d)}}dydx \leq 0.
\end{equation*}
Since by swapping $x\leftrightarrow y$ and making the change of variable $-k\mapsto k$,
\begin{equation*}
\sum_{k}\int_{\mu(t)\leq c}\int_{\mu(t)\leq c}\frac{\mu(t,x)-\mu(t,y)}{|x-y-k|^{d+(2+s-d)}}dydx = \sum_{k}\int_{\mu(t)\leq c}\int_{\mu(t)\leq c}\frac{\mu(t,y)-\mu(t,x)}{|x-y-k|^{d+(2+s-d)}}dydx =0,
\end{equation*}
we conclude that \eqref{eq:rhsL1muc} is $\leq 0$. Hence, $\int_{\T^d}(\mu(t)-c)_{-}dx$ is nonincreasing, and so if $c$ is chosen such that $\inf\mu_0\geq c$, then $\int_{\T^d}(\mu(t)-c)_{-}dx=0$ for every $t$ in the lifespan of $\mu$. This implies that $\inf\mu(t)\geq \inf\mu_0$. An analogous argument shows that if $c\geq \bmu$, then $\int_{\T^d}(\mu(t)-c)_{+}dx$ is nonincreasing. In particular, if $c\geq \sup\mu_0$, then $\int_{\T^d}(\mu(t)-c)_{+}dx=0$ on the lifespan of $\mu$, implying $\sup\mu(t)\leq \sup\mu_0$. The same argument can be leveraged to any $0<s<t$, so that we actually have $t\mapsto\inf\mu(t)$ is nondecreasing, and $t\mapsto\mu(t)$ is nonincreasing.

\bigbreak
\noindent \textbf{Proof of $(iv)$.}
We now prove the estimate for the dissipation of the $L^p$ norms. Let $1<p<\infty$, $p\neq m+1$. By integration by parts, one obtains
\begin{align*}
    &\frac{d}{dt}\|\mu(t)\|_{L^p}^p \\
    &= \int_{\T^d} p \mu^{p-1} \div(\mu^{-m}\nab\g\ast\mu ) dx + \ep\int_{\T^d} p\mu^{p-1}\D\mu dx\\
    &=-pm\int_{\T^d} \mu^{p-m-2}\nab\mu\cdot\nab\g\ast\mu dx+p\int_{\T^d} \mu^{p-1-m}\D\g\ast\mu dx-\ep p(p-1) \int_{\T^d} \mu^{p-2}|\nab\mu|^2dx\\
    &\le-\frac{pm}{p-m-1}\int_{\T^d} \nab(\mu^{p-m-1})\cdot \nab\g\ast\mu dx-p\cd \int_{\T^d} \mu^{p-1-m}|\nab|^{s-d+2}(\mu-\bmu) dx\\
    &=-\frac{pm+p(p-m-1)}{p-m-1}\cd \int_{\T^d} \mu^{p-1-m}|\nab|^{s-d+2}(\mu-\bmu) dx \\
    &=-\frac{p(p-1)}{p-m-1}\cd \int_{\T^d} \mu^{p-1-m}|\nab|^{s-d+2}(\mu-\bmu)dx.
\end{align*}
For the case $p=m+1$, the computations give
\begin{align*}
    \frac{d}{dt}\|\mu(t)\|_{L^{m+1}}^{m+1} &\le -m(m+1) \int_{\T^d} \nab\log\mu \cdot \nab\g\ast\mu dx \\
    &= -m(m+1)\cd \int_{\T^d}\log\mu |\nab|^{s-d+2}(\mu-\bmu)dx.
\end{align*}

\bigbreak
\noindent \textbf{Proof of $(v)$.}
Finally, consider the energy dissipation. By integration by parts, compute
\begin{align*}
    \frac{d}{dt}\hal\int_{\T^d}\g\ast\mu\mu dx &= \int_{\T^d}\g\ast\mu \div(\mu^{-m}\nab\g\ast\mu) dx+ \ep \int_{\T^d} \g\ast\mu  \D\mu dx\\
    &=-\int_{\T^d}|\nab\g\ast\mu|^2\mu^{-m} dx+ \ep\int_{\T^d} \D\g\ast\mu  \mu dx .
\end{align*}
The viscous term can be written, using the definition of $\g$ in \eqref{eq:defg}, as 
\begin{align*}
    \ep\int_{\T^d} \D\g\ast\mu  \mu dx  &= -\ep\cd  \int_{\T^d } (-\D)^\frac{s-d+2}{2}(\mu-\bmu)   \mu dx \\
    &= -\ep\cd  \int_{\T^d } (-\D)^\frac{s-d+2}{2}\mu \mu dx.
\end{align*}
We conclude using Lemma \ref{lem:SVI} that this term is $\le 0$, hence can be dropped.
\end{proof}

We now establish a lower barrier for strong solutions.

\begin{prop}\label{prop:lowerbound}
    Let $T>0$ and $\mu_\ep\in C^{1,1}([0,T]\times\T^d)$ be a strong solution to \eqref{eq:approx} such that $\mu_{0,\ep}\ge \ep>0$ on $\T^d$. Then, there exists a constant $C_{d,s}>0$ depending only on $s,d$ such that for all $t>0$ and $x\in\T^d$,
    \begin{equation}\label{eq:lowerbound}
        \mu_\ep(t,x)\ge \Phi_\ep(t),
    \end{equation}
    where $\Phi_\ep$ is the unique solution to
    \begin{equation*}
    \begin{cases}
        \displaystyle\dot \Phi_\ep = C_{d,s} \frac{\bmu_\ep-\Phi_\ep}{\Phi^m_\ep},\\
        \Phi_\ep(0)=\ep.
    \end{cases}
    \end{equation*}
\end{prop}
\begin{remark}\label{remark:lowerbound}
    As already mentioned in the introduction, the scaling law $t^{\frac{1}{1+m}}$ is that of \cite{stan2015transformations}, for type I self-similar solutions (obtained in the range $\frac{md}{s-d+2}<1$). We refer to \Cref{rem:introLowerbound} for a more detailed discussion. Let us only note that this implies for a.e. $t>0$ and $x\in \T^d$,
    \begin{equation}\label{eq:lowerboundexplicit}
        \mu_\ep(t,x)\ge \Phi_\ep(t) \ge \Phi_0(t) \ge C\min\bigg((\bmu_\ep t)^\frac{1}{1+m}, \quad \bmu_\ep\left(1-e^{-C_{d,s}t}\right) \bigg),
    \end{equation}
    for some $C>0$ depending on $s,d,m$.
\end{remark}
\begin{proof}
    Consider a curve $\bar x:[0,T]\to \T^d$ such that $\bar x(t)\in \argmin \mu_\ep(t)$, for all $t>0$. Then, since $t\mapsto\min\mu_\ep(t)$ is nondecreasing, one can appeal to Lemma \ref{lem:mincurve} and obtain, for almost every $t>0$,
    \begin{align*}
        \frac{d}{dt}\min\mu_\ep(t) &= \p_t\mu_\ep(t,\bar x(t))\\
        &=\div (\mu_\ep^{-m}\nab\g\ast\mu_\ep)(t,\bar x(t)) + \ep\D\mu_\ep(t,\bar x(t)) \\
        &\ge \frac{1}{(\min\mu_\ep(t))^m}\D\g\ast\mu_\ep(t,\bar x(t)) \\
        &= -\cd \frac{1}{(\min\mu_\ep(t))^m}(-\D)^\frac{s-d+2}{2}\mu_\ep(t,\bar x(t)). 
    \end{align*}
    Using Lemma \ref{lem:minprinciple} together with the conservation of the mass, we obtain
    \begin{equation*}
        \frac{d}{dt}\min\mu_\ep(t) \ge C \frac{\bmu_\ep-\min\mu_\ep(t)}{(\min\mu_\ep(t))^m },
    \end{equation*} 
    for some $C>0$ depending on $s,d$.
\end{proof}
We now prove that solutions can be instantaneously bounded even though the initial condition is not. This is not obvious since the mobility is degenerate at points where the solution is unbounded. In particular, this is in sharp contrast with the fast diffusion case. 
\begin{prop}\label{prop:upperbound}
    Let $T>0$ and $\mu_\ep\in C^{1,1}([0,T]\times \T^d)$ be a strong solution to \eqref{eq:approx}, with $\mu_{0,\ep}\ge \ep>0$. Define for all $1\le p \le \infty$ the following quantities:
    \begin{equation*}
    \begin{cases}
        \displaystyle \ga := \frac{md}{s-d+2}, &
        \displaystyle \d_p := \frac{d}{p(s-d+2)-md}, \\
        \displaystyle \zeta_p := \frac{(|p-m|+m)(s-d+2)}{(|p-m|+m)(s-d+2)-md}.
    \end{cases}
    \end{equation*}
    
    If $\ga \ge 1$, then for all $\displaystyle p> \ga$, there exists a continuous increasing map $\omega:\R_+\to\R_+$ depending only on $s,d,p$, such that $\omega(0)=0$, and
    \begin{equation}
        \forall t>0, \quad \|\mu_\ep(t)\|_{L^\infty} \le C t^{-\d_p}\|\mu_{0,\ep}\|_{L^p}^{\zeta_p} + \omega( \bmu_\ep),\label{instantboundClogged}
    \end{equation}
    where $C>0$ depends only on $d, s, p$.\footnote{When $d\ge 2$, we have $p>\ga>m$, so that $\displaystyle \zeta_p = \frac{p(s-d+2)}{p(s-d+2)-md}$. Actually, the reason why we have an absolute value of the form $|p-m| + m$ is only for $d=1$, where we could have $p< m$.}
    
    If $\ga <1$, we have
    \begin{equation}
       \forall t>0, \quad  \|\mu_\ep(t)\|_{L^\infty} \le C t^{-\d_1} \bmu_\ep^{\zeta_1} + C \bmu_\ep,
    \end{equation}
    where $C>0$ depends only on $s,d$.
\end{prop}

\begin{remark}
    By interpolation between $L^p$ and $L^\infty$, we obtain for all $\ga<p\le q\le \infty$ and $t>0$,
    \begin{equation}\label{instantLpClogged}
        \|\mu_\ep(t)\|_{L^q}\le C t^{-\d_p(1-\frac{p}{q})}\|\mu_{0,\ep}\|_{L^p}^{\zeta_p(1-\frac{p}{q}) + \frac{p}{q}} +C\omega(\bmu_\ep)^{1-\frac{p}{q}}\|\mu_{0,\ep}\|_{L^p}^\frac{p}{q},
    \end{equation}
    for some $C>0$ depending additionally on $q$.
\end{remark}

\begin{remark}
    When $\ga < 1$, one obtains a regularization result given any initial finite mass measure. If $\ga \ge 1$, we need a priori more integrability on the initial condition. This threshold is the same as the one observed for self-similar solutions in \cite{stan2015transformations}, which we have already discussed.
\end{remark}
\begin{proof}
Again, we drop the $\ep$ subscript for convenience. 

\bigbreak
\textbf{The case $\ga <1$.}
Considering a curve of maximal points $t\mapsto\bar x(t)$, we evaluate the equation at $(t,\bar x(t))$ and obtain, thanks to Lemma \ref{lem:mincurve},
\begin{align*}
    \frac{d}{dt}\|\mu(t)\|_{L^\infty} &= \div (\mu^{-m}\nab\g\ast\mu) (t,\bar x(t)) + \ep\D\mu(t,\bar x(t))\\
    &\le m\ep^{-m-1}|\nab \mu(t,\bar x(t)) ||\nab\g\ast \mu(t,\bar x(t))| + \|\mu(t)\|_{L^\infty}^{-m} \D\g\ast\mu(t,\bar x(t))+ \ep\D\mu(t,\bar x(t)) \\
    &\le \|\mu(t)\|_{L^\infty}^{-m}\D\g\ast\mu(t,\bar x(t)),
\end{align*}
where we have used that $t\mapsto \min \mu(t)$ is nondecreasing together with first and second order conditions at maximum points. Now, appealing to Lemma \ref{prop:pwD} gives either
\begin{equation*}
    \frac{d}{dt}\|\mu(t)\|_{L^\infty} + C\|\mu(t)\|_{L^\infty}^{-m+1+\frac{s-d+2}{d}}\|\mu(t)\|_{L^1}^{-\frac{s-d+2}{d}} \le 0,
\end{equation*}
or
\begin{equation*}
    \|\mu(t)\|_{L^\infty} \le C_d\|\mu(t)\|_{L^1}.
\end{equation*}
Since the $L^1$ norm is nonincreasing, we obtain the result.

\bigbreak
\textbf{The case $\ga \ge 1$.}
We use an iterative argument in order to obtain the required bound. For any $a>1$, combining the $L^p$ dissipation \eqref{ineq:Lpdissipclogged} and the Stroock-Varopoulos inequality \eqref{ineq:SVI} yields
    \begin{equation}
        \|\mu(t)\|_{L^{a}}^a + C\frac{a(a-1)}{(a-m)^2}\int_0^t \int ||\nab|^\frac{s-d+2}{2}\mu^\frac{a-m}{2}|^2dxd\tau \le \|\mu_{0}\|_{L^a}^a.\label{eq:Lqdissip}
    \end{equation}
Given $p>1$, let us recall the Nash-Gagliardo-Nirenberg inequality (see e.g. \cite{Nirenberg59}):\footnote{See a small caveat on this after the proof.}
\begin{equation*}
    \| f \|_{L^r} \le C \||\nab|^\frac{s-d+2}{2} f\|_{L^2}^\frac{1}{1+\al}\|f\|_{L^q}^\frac{\al}{1+\al} + \| f\|_{L^\lambda},
\end{equation*}
    where 
    \begin{equation*}
    \begin{cases}
        \displaystyle \al = \frac{|p-m|+m}{|p-m|}, & \displaystyle  q = \frac{2(|p-m|+m)}{|p-m|},\\
        \displaystyle r = \frac{d(2|p-m| + m)}{|p-m|(d-\frac{s-d+2}{2})}>1, & \displaystyle \la = \frac{2}{|p-m|}.
    \end{cases}
    \end{equation*}
    Note that in the sequel, the constant $C$ will be uniformly bounded in the iteration procedure. Combine this for $f= \mu^\frac{|p-m|}{2}$ with the $L^p$ dissipation \eqref{eq:Lqdissip} with $a=|p-m|+m$ to obtain, 
    \begin{equation*}
       C\frac{(|p-m|+m)(|p-m|+m-1)}{(p-m)^2} \int_0^t \frac{\left( \|\mu^\frac{|p-m|}{2} \|_{L^r}- \|\mu^\frac{|p-m|}{2}\|_{L^\lambda}\right)^{2(1+\al)} }{\|\mu^\frac{|p-m|}{2}\|_{L^q}^{2\al}} d\tau \le \|\mu_0\|_{L^{|p-m|+m}}^{|p-m|+m}.
    \end{equation*}
    Some simplifications and the conservation of the mass then implies
    \begin{equation*}
         C\frac{(|p-m|+m)(|p-m|+m-1)}{(p-m)^2} \int_0^t \frac{\left( \|\mu \|_{L^\frac{r|p-m|}{2}}^\frac{|p-m|}{2}-  \bmu^\frac{|p-m|}{2}\right)^{2(1+\al)} }{\|\mu\|_{L^{|p-m|+m}}^{|p-m|+m}} d\tau \le \|\mu_0\|_{L^{|p-m|+m}}^{|p-m|+m}.
    \end{equation*}
    We use the decay of all $L^p$ norms to obtain
    \begin{equation*}
        C\frac{(|p-m|+m)(|p-m|+m-1)}{(p-m)^2}  t \left( \|\mu(t) \|_{L^\frac{r|p-m|}{2}}^\frac{|p-m|}{2}-  \bmu^\frac{|p-m|}{2}\right)^{2(1+\al)}\le \|\mu_0\|_{L^{|p-m|+m}}^{2(|p-m|+m)}.
    \end{equation*}
    Rearranging this expression gives
    \begin{align*}
        &\|\mu(t) \|_{L^\frac{r|p-m|}{2}}^\frac{|p-m|}{2}\\
        &\le \bmu^\frac{|p-m|}{2}+ \left( C\frac{(p-m)^2}{(|p-m|+m)(|p-m|+m-1)}t^{-1}\|\mu_0\|_{L^{|p-m|+m}}^{2(|p-m|+m)}\right)^\frac{1}{2(1+\al)} \\
        &= \bmu^\frac{|p-m|}{2}+ \left( C\frac{(p-m)^2}{(|p-m|+m)(|p-m|+m-1)}\right)^\frac{1}{2(1+\al)}t^{-\frac{|p-m|}{2(m+2 |p-m|)}}\|\mu_0\|_{L^{|p-m|+m}}^{\frac{(|p-m|+m)|p-m|}{ 2|p-m|+m}}.
    \end{align*}
    Now, recall the definition of $\lambda = \frac{2}{|p-m|}$ and use $(a+b)^\lambda \le 2^{\lambda-1}a^\lambda +2^{\lambda-1}b^\lambda$ to obtain
    \begin{multline*}
        \|\mu(t) \|_{L^\frac{r|p-m|}{2}} \le 2^{\lambda-1}\bmu \\
        + 2^{\lambda-1} \left(C\frac{(p-m)^2}{(|p-m|+m)(|p-m|+m-1)}\right)^{\frac{1}{2|p-m|+m}} t^{-\frac{1}{2|p-m|+m}} \|\mu_0\|_{L^{|p-m|+m}}^\frac{2(|p-m|+m)}{2|p-m|+m}.
    \end{multline*}
    The same argument works if we consider a sequence of times $t_k = (1-2^{-k})t$ and apply this on $(t_k, t_{k+1})$ instead of $(0,t)$. We then define
    \begin{equation*}
    \begin{cases}
        \displaystyle \tilde \gamma := \frac{d}{d-\frac{s-d+2}{2}}>1, \\
        \displaystyle 2p_{k+1} = \tilde \gamma (2|p_k-m|+m),
    \end{cases}
    \end{equation*}
    for some $p_0>\ga$. Therefore, we can easily see that for $d\ge 2$, we have $p_k>m$ for all $k\ge 0$, and $(p_k)_k$ is increasing since $p_0>\ga$. If $d=1$, then either $p_0>m$, and thus $p_k>m$ for all $k\ge 0$, or $p_0<m$, and then $p_k>m$ for all $k\ge 1$. In every cases, we have $p_k = \ga + \tilde \gamma^{k}(|p_0-m|+m-\ga)$, and $(p_k)_k$ is increasing. We thus arrive at
    \begin{multline*}
        \|\mu(t_{k+1}) \|_{L^{p_{k+1}}}\le 2^{\lambda_k-1} \bmu \\
        + 2^{\lambda_k-1} \left(C\frac{(p_k-m)^2}{(|p_k-m|+m)(|p_k-m|+m-1)}\right)^{\frac{\tilde \gamma}{2p_{k+1}}} 2^{\frac{\tilde \gamma (k+1)}{2p_{k+1}}}t^{-\frac{\tilde \gamma}{2p_{k+1}}} \|\mu(t_k)\|_{L^{p_k}}^\frac{\tilde \gamma (|p_k-m|+m)}{p_{k+1}}.
    \end{multline*}
    Now, note that $\lambda_k\sim_{k\to\infty} C_{p_0}\tilde \gamma^{-k}$ and $$\displaystyle \sup_k \frac{(p_k-m)^2}{(|p_k-m|+m)(|p_k-m|+m-1)} <\infty,$$ so that there is a constant $C>0$ depending additionally on $p_0$ such that
    \begin{equation*}
        \|\mu(t_{k+1}) \|_{L^{p_{k+1}}}\le C \bmu + C^{\frac{\tilde \gamma k}{2p_{k+1}}} t^{-\frac{\tilde \gamma}{2p_{k+1}}} \|\mu(t_k)\|_{L^{p_k}}^\frac{\tilde \gamma (|p_k-m|+m)}{p_{k+1}}.
    \end{equation*}
    This is the starting point of an iteration procedure which gives
    \begin{equation*}
        \|\mu(t_k)\|_{L^{p_k}} \le C^{\eta_k} t^{-\d_k} \|\mu_0\|_{L^{p_0}}^{\zeta_k} + C^{\eta_k}\bmu^{\eta_k} t^{\frac{1}{2p_k}},
    \end{equation*}
    where
    \begin{equation*}
    \begin{cases}
        \displaystyle \eta_k := \frac{1}{2p_k}\sum_{j=1}^{k-1}\tilde \gamma^{j}(k-j), & \displaystyle \d_k := \frac{1}{2p_k}\sum_{j=1}^k \tilde \gamma^j, \\
        \displaystyle \zeta_k := \tilde \gamma^k \frac{|p_0-m|+m}{p_k}.
    \end{cases}
    \end{equation*}
    Sending $k\to\infty$ gives the result.
\end{proof}

A caveat on the proof concerns our use of the Gagliardo-Nirenberg inequality for $L^p$ spaces which are not Banach spaces ($p<1$). This is resolved from the classical Nash-Gagliardo-Nirenberg interpolation inequality for zero mean functions on the torus:
\begin{equation*}
    \| f-\bar f\|_{L^r}\le C\||\nab|^\gamma f\|_{L^2}^\frac{1}{1+\al} \| f\|_{L^q}^\frac{\al}{1+\al}.
\end{equation*}
We then deduce the following:
\begin{equation*}
     \| f\|_{L^r}\le C\||\nab|^\gamma f\|_{L^2}^\frac{1}{1+\al} \| f\|_{L^q}^\frac{\al}{1+\al} + \| f\|_{L^1}.
\end{equation*}
We now interpolate, writing $f=f^\frac{\lambda}{a}f^{1-\frac{\lambda}{a}}$ and use H\"older's inequality, to obtain that $\| f\|_{L^1} \le \| f\|_\lambda^\frac{\lambda}{a} \| f\|_{r}^\frac{r }{b}$, with $1 = \frac{1}{a} + \frac{1}{b}$ and $1 = \frac{\lambda}{a} + \frac{r}{b}$. Then, use $xy\le \frac{1}{\al} x^\al + \frac{1}{\be}y^\be$ for $1=\frac{1}{\al} + \frac{1}{\be}$ to obtain
\begin{equation*}
    \| f\|_1 \le \frac{1}{\al} \| f\|_{\lambda}^\frac{\lambda\al}{a} + \frac{1}{\be} \| f\|_r^{\be(1-\frac{\lambda}{a})}.
\end{equation*}
Even though $\lambda\in (0,1)$, we can thus take $\be = (1-\frac{\lambda}{a})^{-1} > 1$ and conclude as follows:
\begin{align*}
    \| f\|_{L^r}\le C\||\nab|^\gamma f\|_{L^2}^\frac{1}{1+\al} \| f\|_{L^q}^\frac{\al}{1+\al} +\frac{\lambda}{a} \| f\|_\lambda + \frac{1}{\be}\| f\|_{L^r},
\end{align*}
which implies
\begin{equation*}
     \| f\|_{L^r}\le C\||\nab|^\gamma f\|_{L^2}^\frac{1}{1+\al} \| f\|_{L^q}^\frac{\al}{1+\al} + \| f\|_\lambda.
\end{equation*}

We must now provide a weak continuity argument in order for weak solutions to capture the initial datum. This fact is the reason for restricting ourselves to finite energy initial data, though this assumption could be somehow relaxed.

\begin{prop}[Weak continuity]\label{prop:continuity}
    Let $T>0$. Consider a strong solution $\mu_\ep\in C^{1,1}([0,T]\times \T^d)$ to \eqref{eq:approx}. Then, for any $\varphi\in C^{1,1}(
    [0,T]\times\T^d)$ and any $0\le s\le t$,
    \begin{multline*}
        \left\vert\int_{\T^d}\mu_\ep(t,x)\varphi(t,x)dx -\int_{\T^d}\mu_\ep(s, x)\varphi(s,x)dx \right\vert \\
        \le (\|\p_t\varphi\|_\infty \|\mu_{0,\ep}\|_{L^1} + \|\nab\varphi\|_\infty \|\mu_{0,\ep}\|_{\dot H^\frac{s-d}{2}}) \ \omega(t-s)  + \ep \|\D\varphi\|_\infty \|\mu_{0,\ep}\|_{L^1} (t-s),
    \end{multline*}
    for some modulus of continuity $\omega:\R_+\to \R_+$ depending on $m$.
\end{prop}
\begin{remark}
    An inspection of the proof shows that this weak continuity holds not only for the solution we construct, but for any weak solution to \eqref{eq:clogged}.
\end{remark}
\begin{proof}
    We remove the $\ep$ subscript for convenience. Write the equation in its weak form:
    \begin{multline*}
        \int_{\T^d}\mu(t,x)\varphi(t,x)dx -\int_{\T^d}\mu(s,x)\varphi(s,x)dx \\=\int_s^t\int_{\T^d}\p_t \varphi(\tau,x) \mu(\tau,x) dxd\tau - \int_s^t\int_{\T^d}\nab\varphi(\tau,x)\cdot \nab\g\ast\mu(\tau,x)\mu^{-m}(\tau,x) dxd\tau \\
        + \ep \int_s^t\int_{\T^d} \D\varphi(\tau,x)\ \mu(\tau,x) dxd\tau.
    \end{multline*}
    Hence,
    \begin{align*}
        &\left\vert\int_{\T^d}\mu(t,x)\varphi(t,x)dx -\int_{\T^d}\mu(s,x)\varphi(s,x)dx\right\vert \\
        &\le \|\p_t\varphi\|_\infty\|\mu_0\|_{L^1} (t-s) + \|\nab\varphi\|_\infty \int_s^t \int_{\T^d}|\nab\g\ast\mu|\mu^{-m}dxd\tau + \ep \|\D\varphi\|_\infty \|\mu_0\|_{L^1} (t-s)\\
        &\le \|\p_t\varphi\|_\infty\|\mu_0\|_{L^1} (t-s) + \|\nab\varphi\|_\infty \left(\int_s^t\int_{\T^d}|\nab\g\ast\mu|^2\mu^{-m}dxd\tau\right)^\hal\left(\int_s^t\int_{\T^d} \mu^{-m} dxd\tau \right)^\hal \\
        &\quad  + \ep \|\D\varphi\|_\infty \|\mu_0\|_{L^1} (t-s).
    \end{align*}
    Now, use the energy dissipation \eqref{ineq:energydissipationClogged} and the lower bound from Proposition \ref{prop:lowerbound} to obtain
    \begin{align*}
        &\left\vert\int_{\T^d}\mu(t,x)\varphi(t,x)dx -\int_{\T^d}\mu(s,x)\varphi(s,x)dx\right\vert \\
        &\le \|\p_t\varphi\|_\infty\|\mu_0\|_{L^1} (t-s) + \|\nab\varphi\|_\infty \left( \int_{\T^d}g\ast\mu^s\mu^s\right)^\hal \left( \int_s^t\Phi_\ep^{-m}d\tau\right)^\hal +  \ep \|\D\varphi\|_\infty \|\mu_0\|_{L^1} (t-s)\\
        &= \|\p_t\varphi\|_\infty\|\mu_0\|_{L^1} (t-s) + \|\nab\varphi\|_\infty \|\mu_0\|_{\dot H^{\frac{s-d}{2}}} \left( \int_s^t\Phi^{-m}_\ep d\tau\right)^\hal + \ep \|\D\varphi\|_\infty \|\mu_0\|_{L^1} (t-s).
    \end{align*}
    Since $\Phi^{-m}_\ep$ is integrable uniformly in $\ep$, this gives the result.
\end{proof}

Now that all the a priori estimates have been obtained, we can build weak solutions to \eqref{eq:clogged} for any $m>0$. There exists $\mu_{\ep}\in C^{1,1}([0,\infty)\times\T^d)$ which solves the approximate problem \eqref{eq:approx} and satisfies the a priori estimates obtained so far. In particular, we have from Remark \eqref{remark:lowerbound},
\begin{equation*}
    \mu_\ep(t,x) \ge C\min\bigg((\bmu_\ep t)^\frac{1}{1+m}, \bmu_\ep (1-e^{-C_{d,s}t})\bigg).
\end{equation*}
Given $\varphi\in C^{1,1}([0,\infty)\times\T^d)$, we have for all $0<\tau< T$,
\begin{multline*}
    \int_{\T^d}\phi(T) \mu_\ep(T) dx -\int_{\T^d}\phi(\tau) \mu_\ep(\tau)dx-\int_\tau^T\int_{\T^d} \p_t\phi \mu_{\ep} dx dt + \int_\tau^T \int_{\T^d}\nab\phi\cdot \nab\g\ast\mu_{\ep}\mu_{\ep}^{-m}dxdt \\
    =  \ep\int_\tau^T\int_{\T^d} \D\phi \mu_{\ep} dxdt.
\end{multline*}
Passing to the limit in the linear terms is easy. For the nonlinear term, one must provide strong compactness as follows. Considering $\mu_0\in L^p$ for some arbitrary\footnote{We recall that $\displaystyle \ga := \frac{md}{s-d+2}$.} $p>\ga$ ($\mu_0$ is only a finite mass measure when we can take $p=1$, that is when $\ga <1$), we obtain thanks to the $L^p$ dissipation \eqref{ineq:Lp+SVI} and the regularization \eqref{instantLpClogged} that, for all $0<\tau<T$ and $q\in [p,\infty)$,
\begin{equation*}
    C\frac{q(q-1)}{(q-m)^2}\int_\tau^T \int_{\T^d} ||\nab|^{\frac{s-d+2}{2}}\mu_{\ep}^{\frac{q-m}{2}}|^2 dxdt \le C \tau^{-\d_p(q-p)}\|\mu_0\|_{L^p}^{\zeta_p(q-p) + p} + C \om(\bmu)^{q-p}\|\mu_0\|_{L^p}^p,
\end{equation*}
and, for $q\in (1,p)$,
\begin{equation*}
     C\frac{q(q-1)}{(q-m)^2}\int_\tau^T \int_{\T^d} ||\nab|^{\frac{s-d+2}{2}}\mu_{\ep}^{\frac{q-m}{2}}|^2 dxdt \le \|\mu_0\|_{L^q}^q.
\end{equation*}
Together with Proposition \ref{prop:upperbound}, we obtain that $(\mu^\frac{q-m}{2}_{\ep})_{\ep>0}$ is bounded in $L^2((\tau,T), H^\frac{s-d+2}{2}( \T^d))$, for all $q>m$. In particular, taking $q=m+2$ gives that $(\mu_\ep)_\ep$ is bounded in $L^2((\tau,T), H^\frac{s-d+2}{2}( \T^d))$. Therefore, one can extract some $\mu$ satisfying all the a priori estimates and such that $\mu_\ep\xrightarrow[\ep\to 0]{} \mu$ a.e. on $(\tau, T)\times \T^d$ and in $L^2((\tau,T)\times \T^d)$. In particular, for a.e. $t>0$ and $x\in \T^d$,
\begin{equation*}
    \mu(t,x) \ge C\min\bigg( (\bmu t)^\frac{1}{1+m}, \bmu (1-e^{-C_{d,s}t}) \bigg).
\end{equation*}

We then pass to the limit $\ep\to 0$ in the weak formulation above, and then appeal to the weak continuity result of Proposition \ref{prop:continuity} in order to pass to the limit $\tau\to 0$, which concludes the proof.

\section{Lower barrier for fast diffusions.}

In this section, we derive a lower barrier for fast diffusion equations, namely
\begin{equation*}
\begin{cases}
    \p_t \mu -\div (\mu^m \nab \g\ast\mu) = 0, \\
    \mu\vert_{t=0} = \mu_0 \ge 0,
\end{cases}
\end{equation*}
where $0<m<1$ and $\mu_0$ is a Borel measure of finite mass. The existence theory for this equation has been studied in \cite{stanTesoVazquez2019existence}. In the case $m\in (0,1)$, we generically expect infinite speed of propagation. However, such a property could only be proved for $d=1$ (or considering radial solutions in arbitrary dimension), where one can appeal to comparison principles in order to derive lower barriers (see \cite{STAN2016infinitespeed}). Though our method is a priori restricted to the torus, it allows us to derive lower barriers in any dimension, for non radial solutions. 

Again, we provide a priori estimates on the approximate problem
\begin{equation}\label{eq:approxFD}
\begin{cases}
    \p_t\mu_\ep -\div (\mu^m_\ep \nab\g\ast\mu_\ep) = \ep\D\mu_\ep, \\
    \mu_\ep\vert_{t=0} = \mu_{0,\ep},
\end{cases}
\end{equation}
where $\mu_{0,\ep} \ge \ep >0$ is smooth and $\mu_{0,\ep} \xrightharpoonup[\ep\to 0]{} \mu_0$ in the sense of measures. 

For the sake of completeness, we provide known a priori estimates, which can be found in \cite{stanTesoVazquez2019existence}.

\begin{prop}\label{prop:classicalaprioriFD}
    Let $T>0$ and $\mu_\ep\in C^{1,1}([0,T]\times \T^d)$ be a nonnegative solution to \eqref{eq:approxFD}. Then, the following estimates hold for all $t\ge 0$.
    \begin{enumerate}[label=$\left(\roman*\right)$]
        \item (decrease of the $L^1$ norm) $$\displaystyle\|\mu(t)\|_{L^1} \le \|\mu_{\ep,0}\|_{L^1},$$ 
        \item (conservation of mass)
        $$\displaystyle\int_{\T^d}\mu_\ep(t)dx = \int_{\T^d}\mu_{\ep,0}dx =: \bmu_{\ep},$$ 
        \item (weak maximum principle) $t\mapsto \min\mu_\ep(t)$ is nondecreasing and $t\mapsto \max\mu_\ep(t)$ is nonincreasing,
        \item (dissipation of the $L^p$ norms), for all $p\in(1, \infty)$,    
        \begin{equation}
        \displaystyle\|\mu_\ep(t)\|_{L^p}^p + \frac{p(p-1)}{p+m-1}\cd \int_0^t \int_{\T^d} \mu_\ep^{p-1+m}|\nab|^{s-d+2} \mu_\ep\ dxd\tau \le \|\mu_{\ep,0}\|_{L^p}^p\label{ineq:LpdissipFD},
        \end{equation}
        
        \item  (energy dissipation)  
        \begin{equation*}
            \displaystyle\int_{\T^d} \g\ast\mu_\ep(t)\mu_\ep(t)dx + \int_0^t\int_{\T^d}|\nab\g\ast\mu_\ep|^2 \mu_\ep^{m}\ dxd\tau \le \int_{\T^d}\g\ast\mu_{\ep,0}\mu_{\ep,0}dx.
        \end{equation*}
        
    \end{enumerate}
\end{prop}
\begin{remark}
    Combining the $L^p$ dissipation estimate \eqref{ineq:LpdissipFD} with the Stroock-Varopoulos inequality from Lemma \ref{lem:SVI} yields, for all $p\in (1,\infty)$ and all $t\ge 0$,
    \begin{equation}\label{ineq:LpFD+SVI}
        \|\mu_\ep(t)\|_{L^p}^p + \frac{p(p-1)}{(p+m)^2}\cd \int_0^t\int_{\T^d} ||\nab|^\frac{s-d+2}{2}\mu_\ep^\frac{p+m}{2}|^2dxd\tau \le \|\mu_{\ep,0}\|_{L^p}^p. 
    \end{equation}
\end{remark}
\begin{proof}
   Most of these a priori estimates are known, and we refer to \cite[Theorem 2.2]{stanTesoVazquez2019existence}. Let us make a precision concerning what we choose to call ``weak maximum principle''. The proof works as in the clogged diffusion case, and we drop the $\ep$ subscript for clarity. First, if $\mu_0 \geq 0$, then $\mu(t) \geq 0$ on its lifespan. Indeed, let $\mu_+(t), \mu_-(t)$ denote the positive and negative parts of $\mu(t)$ respectively. Then
\begin{equation*}
    \frac{d}{dt}\int_{\T^d} \mu_\pm(t) dx = \hal\frac{d}{dt}\int_{\T^d} \left(|\mu(t)| \pm \mu(t) \right)dx \leq 0,
\end{equation*}
where the final inequality follows from mass conservation and nonincreasingness of the $L^1$ norm. In particular, if $\int_{\T^d} (\mu_0)_- dx =0$, then $0\le \int_{\T^d} \mu_-(t) dx \le \int_{\T^d} (\mu_0)_- dx =0$, which implies that $\mu(t)_- = 0$. Accordingly, if $\mu_0 \le 0$, then $\mu(t) \le 0$ on its lifespan.

Let us now prove that $\inf\mu(t)\ge \inf\mu_0$ and $\sup\mu(t) \le \sup\mu_0$ for all $t\in (0, T)$. Consider $0\leq c\leq \int_{\T^d}\mu_0 dx$ and use integration by parts to compute
\begin{align}
\frac{d}{dt}\int_{\T^d}(\mu-c)_{-}dx &= -\int_{\mu(t)\leq c}(\div(\mu^{m}\nabla\g\ast\mu) +\ep\D\mu)dx \nn\\
&=- \int_{\mu(t)\leq c}\nab\mu^{m}\cdot\nabla\g\ast\mu dx + \cd\int_{\mu(t)\leq c}\mu^{m}\Dm^{2+s-d}(\mu-\bmu) dx \nn\\
&\quad -\ep\int_{\mu(t)=c}\nabla\mu\cdot \frac{\nabla\mu}{|\nabla\mu|}dx \label{eq:dtmucrhsFD}.
\end{align}
The last term is negative and can be discarded. For the first term, observe that $\nabla\mu^{m}\indic_{\mu(t)\leq c} = \nabla(\mu^{m}-c^{m})_{+}$ a.e. Hence, integrating by parts,
\begin{align*}
-\int_{\mu(t)\leq c}\nabla\mu^{m} \cdot\nabla\g\ast\mu = -\cd\int_{\T^d}(\mu^{m}-c^{m})_+\Dm^{2+s-d}(\mu-\bmu)dx.
\end{align*}
Similarly, $\mu^{m}\indic_{\mu\leq c} = (\mu^{m}-c^{m})_{+} +c^{m}\indic_{\mu\leq c}$, which implies that the right-hand side of \eqref{eq:dtmucrhsFD} is equal to
\begin{align}\label{eq:rhsL1mucFD}
\cd c^{m}\int_{\mu(t)\leq c}\Dm^{2+s-d}(\mu-\bmu)dx.
\end{align}
We develop this into
\begin{align*}
\int_{\mu(t)\leq c}\Dm^{2+s-d}(\mu -\bmu) dx = C_{s,d}\sum_{k\in\Z^d} \int_{\mu(t)\leq c}\int_{\T^d} \frac{\mu(t,x)-\mu(t,y)}{|x-y-k|^{d+(2+s-d)}}dydx.
\end{align*}
Evidently,
\begin{equation*}
\int_{\mu(t)\leq c}\int_{\mu(t)>c}\frac{\mu(t,x)-\mu(t,y)}{|x-y-k|^{d+(2+s-d)}}dydx \leq 0.
\end{equation*}
Since by swapping $x\leftrightarrow y$ and making the change of variable $-k\mapsto k$,
\begin{equation*}
\sum_{k}\int_{\mu(t)\leq c}\int_{\mu(t)\leq c}\frac{\mu(t,x)-\mu(t,y)}{|x-y-k|^{d+(2+s-d)}}dydx = \sum_{k}\int_{\mu(t)\leq c}\int_{\mu(t)\leq c}\frac{\mu(t,y)-\mu(t,x)}{|x-y-k|^{d+(2+s-d)}}dydx =0,
\end{equation*}
we conclude that \eqref{eq:rhsL1mucFD} is $\leq 0$. Hence, $\int_{\T^d}(\mu(t)-c)_{-}dx$ is nonincreasing, and so if $c$ is chosen such that $\inf\mu_0\geq c$, then $\int_{\T^d}(\mu(t)-c)_{-}dx=0$ for every $t$ in the lifespan of $\mu$. This implies that $\inf\mu(t)\geq \inf\mu_0$. An analogous argument shows that if $c\geq \bmu$, then $\int_{\T^d}(\mu(t)-c)_{+}dx$ is nonincreasing. In particular, if $c\geq \sup\mu_0$, then $\int_{\T^d}(\mu(t)-c)_{+}dx=0$ on the lifespan of $\mu$, implying $\sup\mu(t)\leq \sup\mu_0$. The same argument can be leveraged to any $0<s<t$, so that we actually have $t\mapsto\inf\mu(t)$ is nondecreasing, and $t\mapsto\mu(t)$ is nonincreasing.
\end{proof}
Let us now state the lower barrier result we can prove for this equation. This shows in particular infinite speed of propagation on a bounded domain.
\begin{prop}
    Let $T>0$ and $\mu_\ep \in C^{1,1}([0,T]\times \T^d)$ be a solution to \eqref{eq:approxFD} such that $\mu_{0,\ep} \ge \ep >0$. Then, for all $t>0$ and $x\in\T^d$,
    \begin{equation*}
        \mu_\ep(t,x) \ge \Phi_\ep(t),
    \end{equation*}
    where $\Phi_\ep$ is the unique \emph{positive} solution to
    \begin{equation*}
    \begin{cases}
        \dot \Phi_\ep = C_{d,s}\Phi_\ep^m\left(\bmu_\ep-\Phi_\ep\right), \\
        \Phi_\ep(0) = \ep.
    \end{cases}
    \end{equation*}
\end{prop}
\begin{remark}
    This implies for almost every $t>0$ and $x\in \T^d$,
    \begin{equation*}
        \mu_\ep(t,x)\ge \Phi_\ep(t) \ge \Phi_0(t) \ge  C\min\bigg((\bmu t)^\frac{1}{1-m}, \quad \bmu\left(1-e^{-C_{d,s}t}\right) \bigg),
    \end{equation*}
    for some $C>0$ depending on $s,d,m$. The scaling law $t^{\frac{1}{1-m}}$ is again the same as that obtained in \cite{stan2015transformations}.
\end{remark}
\begin{remark}
    Notice that the ODE satisfied by $\Phi_0$ is a typical example of non-uniqueness. However, the approximation procedure selects the only solution that is positive for all positive times, since $\mu_{0,\ep} \ge \ep >0$.
\end{remark}
\begin{proof} Let us drop the $\ep$ subscript.  
    Consider a curve $\bar x:[0,T]\to \T^d$ such that $\bar x(t)\in \argmin \mu(t)$, for all $t>0$. Then, since $t\mapsto\min\mu(t)$ is nondecreasing, one can appeal to Lemma \ref{lem:mincurve} and obtain for almost every $t>0$,
    \begin{align*}
        \frac{d}{dt}\min\mu(t) &= \p_t\mu(t,\bar x(t))\\
        &=\div (\mu^{m}\nab\g\ast\mu)(t,\bar x(t)) + \ep\D\mu(t)(\bar x(t)) \\
        &\ge (\min\mu(t))^m\D\g\ast\mu(t)(\bar x(t)) \\
        &= -\cd (\min\mu(t))^m(-\D)^\frac{s-d+2}{2}\mu(t)(\bar x(t)). 
    \end{align*}
    Using Lemma \ref{lem:minprinciple} together with the conservation of the mass, we obtain
    \begin{equation*}
        \frac{d}{dt}\min\mu(t) \ge C (\min\mu(t))^m (\bmu-\min\mu(t)),
    \end{equation*} 
    for some $C>0$ depending on $s,d$.
\end{proof}

We finally state the following (well-known) regularization result, though the proof we use is more direct.
\begin{prop}
    Let $T>0$ and $\mu_\ep\in C^{1,1}([0,T]\times\T^d)$ be a solution to \eqref{eq:approxFD}. Then, for all $t>0$, 
    \begin{equation*}
        \|\mu_\ep (t)\|_{L^\infty} \le C\min(1,t)^{-\d}\bmu_\ep, 
    \end{equation*}
    where $\displaystyle \d := \frac{d}{s-d+2 + md}$, and $C>0$ depends on $s,d$.
\end{prop}
\begin{remark}
    By interpolation, one obtains a bound on $L^p$ norms:
    \begin{equation*}
        \forall p\in [1,\infty],\ \forall t>0, \quad \|\mu_\ep(t)\|_{L^p } \le C\min(1,t)^{-\d(1-\frac{1}{p})}.
    \end{equation*}
\end{remark}
\begin{proof}
    Denote $\mu \equiv \mu_\ep$. Considering a curve of maximal points $t\mapsto\bar x(t)$, we evaluate the equation at $(t,\bar x(t))$ and obtain, thanks to Lemma \ref{lem:mincurve}, 
\begin{align*}
    \frac{d}{dt}\|\mu(t)\|_{L^\infty} &= \div (\mu^{m}\nab\g\ast\mu) (t,\bar x(t)) + \ep\D\mu(t,\bar x(t))\\
    &\le \|\mu(t)\|_{L^\infty}^{m}\D\g\ast\mu(t,\bar x(t)).
\end{align*}
Now, appealing to Lemma \ref{prop:pwD} gives either
\begin{equation*}
    \frac{d}{dt}\|\mu(t)\|_{L^\infty} + C\|\mu(t)\|_{L^\infty}^{m+1+\frac{s-d+2}{d}}\|\mu(t)\|_{L^1}^{-\frac{s-d+2}{d}} \le 0,
\end{equation*}
or
\begin{equation*}
    \|\mu(t)\|_{L^\infty} \le C_d\|\mu(t)\|_{L^1}.
\end{equation*}
This implies the result.
\end{proof}
We can now construct a solution to \eqref{eq:FD}, as in \cite[Theorem 2.2]{stanTesoVazquez2019existence}.

\printbibliography

\end{document}